\documentclass[12pt]{article}
\usepackage{amsmath}
\usepackage{amsfonts}
\usepackage{amssymb}
\usepackage{graphicx}
\usepackage{amsthm}
\usepackage{enumerate}
\usepackage{color}
\usepackage{url,xargs}
\usepackage{hyperref}
\newtheorem{theorem}{Theorem}
\newtheorem{lemma}[theorem]{Lemma}

\newtheorem{definition}[theorem]{Definition}

\newtheorem{corollary}[theorem]{Corollary}

\newtheorem{proposition}[theorem]{Proposition}
\newtheorem{problem}[theorem]{Problem}

\begin{document}

\title{The structure of graphs with Circular flow number
5 or more, and the complexity of their recognition problem}

\author{L. Esperet \footnote{Laboratoire G-SCOP (Grenoble-INP, CNRS),
    Grenoble, France. Partially supported by ANR Project Heredia
  (\textsc{anr-10-jcjc-0204-01}), ANR Project Stint
  (\textsc{anr-13-bs02-0007}), and LabEx PERSYVAL-Lab
  (\textsc{anr-11-labx-0025}).} \and
  G. Mazzuoccolo \footnote{Universit\`{a} di Modena e Reggio Emilia, Dipartimento di Scienze Fisiche, Informatiche e Matematiche, Italy e-mail: {mazzuoccolo@unimore.it}}
    \and M. Tarsi \footnote{The Blavatnik School of Computer Science,
    Tel Aviv University, Israel e-mail: {tarsi@post.tau.ac.il}}}

\maketitle

\begin{abstract} \noindent
For some time the Petersen graph has been the only known Snark
with circular flow number $5$ (or more, as long as the assertion
of Tutte's $5$-flow Conjecture is in doubt). Although infinitely
many such snarks were presented eight years ago  in \cite{mr}, the
variety of known methods to construct them and the structure of
the obtained graphs  were still rather limited. We start this
article with  an analysis of sets of flow values, which can be
transferred through flow networks  with the flow on each edge
restricted to the open interval $(1,4)$ modulo $5$. All these sets
are symmetric unions of open integer intervals in the ring
$\mathbb{R}/5\mathbb{Z}$. We use the results to design an arsenal
of methods for constructing snarks $S$  with circular flow number
$\phi_c(S)\ge 5$. As one indication to the diversity and density
of the obtained family of graphs, we show that it is sufficiently
rich so that the corresponding recognition problem is NP-complete.
\medskip

\end{abstract}

\section{Introduction}

For an integer $k$, a \emph{nowhere-zero $k$-flow} in a graph $G$
is a flow in some orientation of $G$, such that the flow value on
each arc is in $\{1,2,\ldots,k-1\}$. In this paper, we will be
mainly interested in the following relaxation of nowhere-zero
flows: for some real $r$, a \emph{circular nowhere-zero $r$-flow}
in a graph $G$ is a flow in some orientation of $G$, such that the
flow value on each arc is in $[1,r-1]$. The \emph{circular flow
number} of a graph $G$ is the infimum of the reals $r$ such that
$G$ has a circular nowhere-zero $r$-flow (more detailed definitions
 are presented in Section~\ref{sec:cnzf}). It was conjectured by
Tutte that any 2-edge-connected graph has a nowhere-zero
5-flow~\cite{5flow}, and that any 4-edge-connected graph has a
nowhere-zero 3-flow~\cite{Tut66}. A stronger conjecture of
Jaeger~\cite{Jae88} asserts that for any integer $k\ge 1$, any
$4k$-edge-connected graph has circular flow number at most
$2+\tfrac1k$. These conjectures are wide open.

A particularly interesting class of graphs for studying the
interplay between edge-connectivity and flows is the class of
\emph{snarks}, which are cyclically 4-edge-connected cubic graphs
of girth at least 5, with no 3-edge-coloring (equivalently, with
no circular nowhere-zero 4-flow). Because a snark is
well-connected, one might expect that it has circular flow number
less than 5. The Petersen graph shows that it is not the case,
however it was conjectured by Mohar in 2003 that it is the only
counterexample~\cite{mohar}. The conjecture was refuted in 2006 by
M\'a\v cajov\'a and Raspaud~\cite{mr}, who constructed an infinite
family of snarks with circular flow number 5.

The purpose of this paper is to show that the family $S_{\ge 5}$ of
snarks with circular flow number at least 5 is significantly richer
than just the set of graphs constructed by M\'a\v cajov\'a and
Raspaud. Observe that if a graph $G$ has circular flow number less
than some value $k$, then $G$ admits a flow such that every arc has
flow value in the open interval $(1,k-1)$. It is known to be
equivalent to the fact that $G$ has a \emph{modulo $k$} flow such that
in some orientation of $G$, every arc has flow value in
$(1,k-1)$. With this observation in mind, we will make a systematic
study of the set of modulo $k$ flow values that can be transferred
through a two terminal network, where the flow on each edge is
restricted to $(1,k-1)$. In Section \ref{sec:ap}, we will show that
for any two terminal network, this set of flow values is a symmetric
union of open integer intervals in the ring $\mathbb{R}/k\mathbb{Z}$
(more will be said about these unions of intervals in Section
\ref{sec:al}). Moreover, the set $\mbox{\it GI}_k$ of unions of
intervals that can be obtained in this way is closed under addition
and intersection.

In Section \ref{>=5} we will focus on the case $k=5$. Using ideas
developed in the previous sections, we construct a certain set of
two terminal networks allowing us to generate most of $\mbox{\it
  GI}_5$. In Section \ref{explicit}, these networks will be pieced
together in several ways, allowing us to design a variety of
methods for the construction of graphs (in particular snarks) of
circular flow number at least $5$. These constructions are not
independent, so the same graphs can be obtained in several
different ways (more will be said about the redundancy of our
constructions in Section~\ref{sec:red}). Consequently, it is at
first sight unclear how rich the constructed family is.

In Section \ref{sec:npc}, we will make this clear by showing that
deciding whether a snark has circular flow number less than 5 is an
NP-complete problem (and so the recognition problem for $S_{\ge 5}$ is
co-NP-complete). We also generalize that result to any circular flow
number $r$, $r \in (4,5]$. The proof uses the tools developed in
  Section \ref{>=5}. Many other nowhere zero flow related problems are known to be
  NP-hard:
  deciding the existence of nowhere zero 3-flow is NP-complete even if restricted to planar graphs
   (where it is equivalent to 3-coloring of the dual graph); deciding the existence of nowhere
    zero 4-flow is NP-complete also when restricted to cubic graphs (where it is equivalent to
    3-edge-coloring);
   when considering nowhere zero 5-flows, it was proved by Kochol~\cite{Koc98} that if Tutte's
    5-flow conjecture
  (mentioned above) is wrong, then deciding whether a cubic graph has
  a nowhere-zero 5-flow is an NP-complete problem. It was also proved
  in~\cite{EMOP13} that for any $t$ and $k$, either all $t$-edge-connected planar
  (multi) graphs have circular flow number at most $2+\tfrac1k$, or
  deciding whether a $t$-edge-connected planar (multi) graph has
  circular flow number at most $2+\tfrac1k$ is an NP-complete problem
  (this is related to the conjecture of Jaeger mentioned
  above).

  Our result, however, is significantly different. Only a limited family of snarks with flow
   number 5 (or more)
  where so far known and here we show that the family of such graphs is rich enough to yield
  NP-hardness. Our result
  does not rely on any unsettled conjecture. If Tutte's conjecture indeed holds, then this is
  the family of snarks whose flow number is  precisely 5.  Another unique aspect of our proof
   is the need to restrict the construction to actually provide
   snarks, unlike other proofs which strongly rely on graphs with
   smaller girth and connectivity.

 In Section \ref{sec:ccl}, we conclude with some open problems and suggestions for further
  research.

\section{The algebra of symmetric unions of open integer
intervals in $\mathbb{R}/k\mathbb{Z}$}\label{sec:al}

\begin{definition} Let $r$ be a positive real number. $\mathbb{R}/r\mathbb{Z}$ denotes the ring of real
numbers modulo $r$. $\mathbb{R}/r\mathbb{Z}$ is commonly
represented by a cycle of length $r$ where the numbers of the real
interval [0,r) are cyclically ordered clockwise  from $0$ to
$r=0$. An interval $(a,b)$ on $\mathbb{R}/r\mathbb{Z}$ refers to
the set of numbers covered  when traversing clockwise from $a$ to
$b$. Closed and half closed intervals are similarly defined. It
follows, for example, that
 $0\in (a,b)$ if and only if, when
referred to as real numbers, $a>b$. Also the union of the disjoint
intervals [a,b) and [b,a) is always the entire cycle $\mathbb{R}/r\mathbb{Z}$.
Traversing clockwise from $a$ to $b$ is ambiguous when $a=b$. We
resolve this by taking the long way, so $(a,a)=\mathbb{R}/r\mathbb{Z} - \{a\}$ (and
not the empty set).
\end{definition}

The following relates to $\mathbb{R}/k\mathbb{Z}$  where $k$ is a
positive integer.

\begin{definition}An {\bf integer open interval} of $\mathbb{R}/k\mathbb{Z}$ is any interval
$(a,b)$  where $a$ and $b$ are (not necessarily distinct)
integers. There are clearly $k^2$ such intervals with $a,b \in
\{0,1,...,k-1\}$. The set of all integer intervals of
$\mathbb{R}/k\mathbb{Z}$ is denoted here by $I_k$.
\end{definition}

\begin{definition}
A set $A \in \mathbb{R}/k\mathbb{Z}$ is {\bf symmetric} if and only if \\ $a \in A
\Leftrightarrow -a(=k-a) \in A$.
\end{definition}
\begin{definition}
Let $\mbox{\it SI}_k$ denote the set of all unions of subsets of
$I_k$  which form symmetric subsets of $\mathbb{R}/k\mathbb{Z}$.
\end{definition}

For example, $0$ and $1$ are the only integers in $\mathbb{R}/2\mathbb{Z}$ and
accordingly, $I_2=\{(0,0),(0,1),(1,0),(1,1)\}$ and
\[\mbox{\it SI}_2=\{\emptyset,~(0,1)\cup(1,0),~(0,0),~(1,1),~\mathbb{R}/2\mathbb{Z}\}\]

All four non-empty sets in $\mbox{\it SI}_2$ contain every
non-integer and they only differ by a different subset of
$\{0,1\}$  that each of them contains. Starting with $|\mbox{\it
SI}_1|=3$ and $|\mbox{\it SI}_2|=5$, it is not hard to verify that
$|\mbox{\it SI}_k|$ are Fibonacci numbers. Coming next are
$|\mbox{\it SI}_3|=8$, $|\mbox{\it SI}_4|=13$ and $|\mbox{\it
SI}_5|=21$.

\begin{proposition}
\label{algebra} $\mbox{\it SI}_k$ is clearly closed under the following set
operations:
\begin{itemize}
\item {\bf Addition}, defined by $A+B=\{a+b\,|\,a\in A \mbox{ and }b\in B\}$.
As the involved sets are symmetric, addition can be replaced by
subtraction (yet $A+A=A-A$ is neither $``0"$, nor empty and
addition is not invertible);
\item Set {\bf Intersection}, $A\cap B$;
\item Set {\bf Union}, $A\cup B$;
\item {\bf Open complement}. The open complement of an open set
$A$ is the complement of its closure $\sigma(A)$, that is,
$\overline{\sigma(A)}=\mathbb{R}/k\mathbb{Z} - \sigma(A)$.
Following that definition, $A \cup \overline{\sigma(A)}$ is
generally not the entire cycle, but lacks some integers.
\end{itemize}

In what follows  addition and intersection play the major role.
\end{proposition}

\section{Applications to the theory of circular nowhere-zero
flows}\label{sec:ap}
\subsection{Circular nowhere-zero flows}\label{sec:cnzf}

 \begin{definition} Given a real number $r\geq 2$, a {\bf circular nowhere-zero
$r$-flow} ($r$-{\sc cnzf} for short) in a graph $G=(V,E)$ is an
assignment $f:E \rightarrow [1,r-1]$ and an orientation $D$ of
$G$, such that $f$ is a {\bf flow} in $D$. That is, for every
vertex $x\in V$, $\sum_{e\in E^+(x)}f(e)=\sum_{e\in E^-(x)}f(e)$
where $E^+(x)$, respectively $E^-(x)$, are the sets of edges
directed from, respectively toward, $x$ in $D$.
\end{definition}
Accordingly defined is:
\begin{definition} The \textbf{circular flow number} $\phi_c(G)$ of a graph $G$ is the
infimum of the set of numbers $r$  for which $G$ admits an
$r$-{\sc cnzf}. If $G$ has a bridge then we define
$\phi_c(G)=\infty$.
\end{definition}

The notion of $r$-{\sc cnzf} was first introduced in \cite{gtz},
observing that $(k,d)$-coloring, previously studied by Bondy and
Hell \cite{BH}, can be interpreted as the dual of real (rather
than integer)-valued nowhere-zero flow. \textbf{Integer}
nowhere-zero flows are much more widely known and intensively
studied  since first presented by W. Tutte \cite{5flow}  60 years
ago. A comprehensive  source for material on integer flows and
related topics is C.Q. Zhang's book \cite{zhang}.

 A circular nowhere-zero {\bf modular}-$r$-flow ($r$-{\sc mcnzf})  is an analogue of an $r$-{\sc cnzf},
 where the additive group of real numbers is replaced by the additive group of $\mathbb{R}/r\mathbb{Z}$.

\begin{definition}
An {\bf $r$-{\sc mcnzf}} in a graph $G=(V,E)$ is an assignment
$f:E \rightarrow [1,r-1] \subseteq \mathbb{R}/r\mathbb{Z}$ and an
orientation $D$ of $G$, such that  for every vertex $x\in V$,
$\sum_{e\in E^+(x)}f(e)=\sum_{e\in E^-(x)}f(e)$. Summation is
performed in $\mathbb{R}/r\mathbb{Z}$.
\end{definition}
Part of the definition of an $r$-{\sc cnzf} is an orientation
where all flow values are positive. Since there are no ``positive"
or ``negative" numbers in $\mathbb{R}/r\mathbb{Z}$, the
orientation where an $r$-{\sc mcnzf} $f$ is defined  is only
required ``for reference". The direction of an edge $e$ can be
reversed and $f$ transformed into another $r$-{\sc mcnzf}, where
$f(e) \in \mathbb{R}/r\mathbb{Z}$ is replaced by $-f(e) \in
\mathbb{R}/r\mathbb{Z}$. As a measure to relate to different
orientations we define:

\begin{definition}
Let $f$ be an $r$-{\sc mcnzf} in a reference orientation $D_0$ of
a graph $G$. We refer by $f_D$ to the flow, on an orientation $D$,
defined by $f_D(e)=f(e)$  for edges $e$ of the same orientation in
$D_0$ and $D$, and $f_D(e)=-f(e)$ if the directions of $e$ in
$D_0$ and $D$  differ.
\end{definition}

 The following result can be easily deduced from
Tutte's original work on integer flows \cite{5flow} and it is also
explicitly stated in some more recent literature, e.g. \cite{pan}

\begin{proposition}\label{mod}
The existence of a circular nowhere-zero $r$-flow in a graph $G$
is equivalent to that of an $r$-{\sc mcnzf}.

Furthermore, if $f$ is an $r$-{\sc mcnzf} in (an orientation
$D_0$) of $G$, then there exists an orientation $D$ of $G$ and a
(real valued positive) $r$-{\sc cnzf} $g$ in $D$, such that  for
every edge $e$ of $G$, $g(e)\equiv f_D(e)$ modulo $r$.
\end{proposition}
Consequently,
\begin{proposition}
\label{modular} For any graph $G=(V,E)$, $\phi_c(G)<r$ if and only if there exists an
$r$-{\sc mcnzf} $f$ in $G$  such that $f:E\rightarrow(1,r-1)$.
Accordingly, we refer to such a flow $f$  as a {\bf sub-$r$-{\sc
mcnzf}}.
\end{proposition}
\begin{proof}
If $\phi_c(G)=r-\epsilon<r$, then there exists a flow $g$ in $G$,
$g:E\rightarrow[1,r-1-\epsilon]$. For a small enough $\delta$,
$(1+\delta)g:E\rightarrow(1,r-1)$. When its values  are
interpreted as modulo $r$ residues, $(1+\delta)g$ yields an
$r$-{\sc mcnzf} with range $(1,r-1)$ modulo $r$, as claimed. On
the other hand, if $f:E\rightarrow(1,r-1)$ is a sub-$r$-{\sc
mcnzf}, then, by Proposition \ref{mod}, there exists an $r$-{\sc
cnzf} $g$  such that $g:E\rightarrow(1,r-1)\subseteq \mathbb{R}$.
Recall that if $t\in (1,r-1)$ modulo $r$, then both $t$ and $r-t$
belong to $(1,r-1)$ as real numbers. Then $g$ is clearly an
$r'$-{\sc cnzf} for some $r'<r$, so $\phi_c(G)<r$.
\end{proof}

\subsection{Generalized edges}
\begin{definition}
A {\bf generalized edge} (g-edge  for short) $G_{u,v}$ is simply a
graph $G=(V,E)$ and two of its vertices  $u$ and $v$. The vertices
$u$ and $v$ are the {\bf terminals} of $G_{u,v}$. Using g-edges,
as well as more generalized similar structures  is a rather common
technique, when dealing with circular NZF's, e.g
\cite{pan,mr,lokut,Hag12}  where g-edges are called ``dipoles",
``two terminal networks" and ``2-poles".
\end{definition}

Our interest in g-edges results from the following parameter:
\begin{definition}\label{capacity} The {\bf open $r$-capacity} $\mbox{\it CP}_r(q)$ of a g-edge
$q=G_{u,v}$ is a subset of $\mathbb{R}/r\mathbb{Z}$, defined as
follows: Add to $G$ an additional edge $e_0\notin E(G)$  with
endvertices $u$ and $v$, and set:
\[\mbox{\it CP}_r(G_{u,v})=\{f(e_0)\,|\,f \mbox{ is a modulo r flow in }G\cup e_o \mbox{ and } f:E(G)\rightarrow (1,r-1)\} \]
\end{definition}

More visually, $\mbox{\it CP}_r(G_{u,v})$ is the set of all flow
values (in $\mathbb{R}/r\mathbb{Z}$) which can be ``pushed"
through $G$ from source $u$ to sink $v$, under all orientations of
$G$, where the ``flow capacity" of every edge of $G$ is restricted
to $(1,r-1)$.

A seemingly more natural parameter is the {\bf closed
$r$-capacity}, where the set of allowed flow values is the entire
closed interval $[1, r - 1]$. The notion of a closed capacity is
derived from the definition of an $r$-MCNZF. It is extensively
used in~\cite{pan} to construct graphs $G$ with $\phi_c(G) = r$,
for any rational $4 < r < 5$.  Despite the similar definitions,
open capacities are not straightforwardly obtained from closed
ones. If the closed capacity of a g-edge contains a closed
interval $[a,b]$, the open capacity of the same g-edge does not
necessarily contain the open interval $(a,b)$, and it might even
be disjoint from $[a,b]$. Closed capacities do not serve our
needs, as we focus on sub-$r$-{\sc mcnzf}'s, in particular
sub-5-{\sc mcnzf}. No g-edge is known of  closed 5-capacity other
than $(1,4)$, or $\mathbb{R}/5\mathbb{Z}$. Considering the 5-flow
conjecture, it might well be the case that none exists. In what
follows we only deal with open capacities and omit the term
``open'' when referring to one.

The notion of g-edge relates to $\mbox{\it SI}_k$ via the following basic property of
$k$-flows:
\begin{lemma}
\label{cpk}
 $\mbox{\it CP}_k(G_{u,v})\in \mbox{\it SI}_k$, for any integer $k\ge 2$ and a g-edge
 $G_{u,v}$.
\end{lemma}

\begin{proof}
Take any $t\in \mbox{\it CP}_k(G_{u,v})$. By definition, there
exists a modulo $k$ flow $f$ in an orientation of $H=G\cup uv$,
such that $f:E(G)\rightarrow(1,k-1)$ and $f(uv)=t$. For a given
$\epsilon>0$ we select $0<\delta<\epsilon/|E(G)|$. We also make
sure that $\delta$ is small enough, so that
$f:E(G)\rightarrow(1+\delta,k-1-\delta)$. It follows that
$f:E(G)\rightarrow [1+\delta,k-1-\delta]$, and, for that closed
interval, no edge $e\in E(G)$ is saturated by $f$, in the sense
that $f(e)$ reaches neither of the bounds $1+\delta$ and
$k-1-\delta$. We now keep the orientation unchanged and apply a
``Max flow" algorithm on the network $G$ with source $u$ and sink
$v$, starting from $f$, in both directions. Although arithmetic is
modulo $k$, the relevant elements of Network flow theory (with
upper and lower capacities) are still valid: Increasing,
respectively decreasing, translates to moving clockwise,
respectively counterclockwise. An augmenting path is a path from
$v$ to $u$, where the current flow on an edge can be increased by
pushing it toward  $k-1-\delta$ if the edge  has the direction of
the path, or toward $1+\delta$ if the edge is of the opposite
direction. Starting with the flow $f$, the flow on $uv$ can be
{\bf continuously} augmented (say, by repeatedly selecting a
shortest augmenting path, to guarantee termination), until a
saturated edge-cut is reached. Similarly, $f$ can be continuously
decreased along decreasing paths (from $u$ to $v$, where the roles
of the bounds $k-1-\delta$ and $1+\delta$ are switched), until a
cut is saturated in the opposite direction. In a saturated cut
obtained by augmenting $f(uv)$, the flow values are $k-1-\delta$
on $m_1$ edges, directed forward, and $1+\delta$ on $m_2$ edges,
directed backwards. The ``max" (modulo $k$) flow on $uv$ at that
stage is $m_1(k-1-\delta)-m_2(1+\delta) =
-(m_1+m_2)-(m_1+m_2)\delta $. Clearly $a=-(m_1+m_2)$ is an integer
and $(m_1+m_2)\le |E(G)|$, so $(m_1+m_2)\delta<\epsilon$. In this
manner, the flow $t$ on $uv$ can be continuously increased to $a -
\epsilon$ where $a$ is an integer, and $\epsilon$ any small
positive number. Similarly, it can be continuously decreased from
$t$ to $b+\epsilon$ for an integer $b$. In summary, for every
small $\epsilon$, there exist integers $a$ and $b$, such that
$t\in (b+\epsilon,a-\epsilon)\subseteq \mbox{\it CP}_k(G_{u,v})$.
This implies that, if $t$ is not an integer, then the entire open
unit interval that includes $t$ is contained in $\mbox{\it
CP}_k(G_{u,v})$. If $t$ is an integer then both open unit
intervals on its two sides are contained in $\mbox{\it
CP}_k(G_{u,v})$. It follows that $\mbox{\it CP}_k(G_{u,v})$ is
indeed a union of open integer intervals. Note that the value $-t$
is obtained by the flow $-f$, so $\mbox{\it CP}_k(G_{u,v})$ is
also symmetric.
\end{proof}
 In what follows we allow the edge set of a graph to consist of both
 (simple) edges and g-edges, by means of the following convention: A
 g-edge $q=H_{u,v}$ in a graph $G$ is a subgraph $H$ of $G$ that
 shares its terminals $u$ and $v$ with the rest of the graph, and is
 otherwise vertex-disjoint from $G-H$. If $f$ is a flow in $G$,
 then $f(q)$ denotes the amount of flow that ``traverses from $u$ to
 $v$ (or the other way around) through the subgraph $H$".  When
 considering sub-$k$-{\sc mcnzf}'s, the main characteristic of a g-edge is
 its $k$-capacity. With that in mind, we refer to a g-edge $q$ with
 terminals $u$ and $v$ and $k$-capacity $A \in \mbox{\it SI}_k$ as an $A$-edge,
 $q=uv$ without elaborating any further on its structure. For that
 matter, a $(1,k-1)$-edge, may, or may not be a simple edge. Any
 other capacity implies a genuine generalized edge. A flow $f$ is a
 sub-$k$-{\sc mcnzf} if and only if for every edge $e$, simple or
 generalized, $f(e) \in \mbox{\it CP}_k(e)$.

 \begin{definition}
 We say that a set $A \in \mbox{\it SI}_k$ is {\bf graphic}, if there exists a
 g-edge with $k$-capacity $A$. The set of all graphic members of
 $\mbox{\it SI}_k$ is denoted here by $\mbox{\it GI}_k$.
 \end{definition}
Rather obvious, yet fundamental observations are:
\begin{proposition}\label{parser}
\noindent \begin{itemize}
 \item Let $q$ with $k$-capacity $A$ and $t$ with $k$-capacity
  $B$  be two  g-edges, sharing a pair of  terminals
  $u$ and $v$ and otherwise disjoint. The union of $q$ and $t$ is called the {\bf parallel join} of $q$
  and $t$ and it forms a new g-edge with terminals $u$ and $v$
  and $k$-capacity  $A+B$ (see Proposition~\ref{algebra} for the definition of $A+B$).

 \item The union of two g-edges $q=uv$ with $k$-capacity $A$, and $t=vw$ with
 $k$-capacity $B$, which share a single terminal $v$ forms the {\bf serial join} of $q$ and
 $t$. That is a new g-edge with terminals $u$ and $w$ and $k$-capacity $A\cap B$.

 \item It follows that the subset $\mbox{\it GI}_k$ of all graphic members of $\mbox{\it SI}_k$ is
 closed under Addition and Intersection and as such, it is a
 {\bf sub-algebra} of $\mbox{\it SI}_k$ with respect to these two operations.

 \end{itemize}
\end{proposition}

 Let us demonstrate the above by an analysis of the algebra
 $\mbox{\it GI}_3$.
 The $k$-capacity of a simple edge is, by definition, $(1,k-1)$.
 For $k=3$, that is $(1,2)$. The $(+,\cap)$-algebra generated by
 $(1,2)$ includes the following $6$ members:
 \begin{itemize}
\item
$(1,2)$\item$(2,1)=(1,2)+(1,2)$\item$(0,0)=(1,2)+(2,1)$
\item$\mathbb{R}/3\mathbb{Z}=(0,0)+(1,2)$\item$\emptyset=(1,2)\cap(2,1)$
\item$(2,0)\cup(0,1)=(0,0)\cap(2,1)$

\end{itemize}

We tend to believe that the remaining two sets in $\mbox{\it SI}_3$, namely
$\mathbb{R}/3\mathbb{Z}-\{1,2\}$ and
$\mathbb{R}/3\mathbb{Z}-\{0,1,2\}$ are not graphic, yet, at this
point, we have no serious evidence to support such a claim.

\section{$\mbox{\it GI}_5$ and some related observations}\label{>=5}
\begin{definition}
Associated with a set $A \in \mbox{\it SI}_k$ are two size
parameters: its {\bf amplitude} $\mbox{Am}(A)$, which is the
length (number of unit intervals) of the smallest interval that
contains $A$, and its {\bf measure} $\mbox{Me}(A)$, which is the
number of unit intervals contained in $A$.
\end{definition}
In this section, the capacity of a g-edge refers to its open
$5$-capacity
\subsection{Generating  $\mbox{\it GI}_5$}

Included in $\mbox{\it GI}_5$ are
\begin{itemize}
\item $(1,4)$ of amplitude and measure $3$, represented by a
simple edge. We will later build additional $(1,4)$-edges, to
serve some needs (related mostly to edge-connectivity), of
specific constructions.
\item
$\mathbb{R}/5\mathbb{Z}=(1,4)+(1,4)$ of amplitude and measure $5$
\end{itemize}
These two sets form a closed sub-algebra, so another generator is
required in order to go further. Such a generator is the
$5$-capacity of ${\cal P}^*_{10}(u,v)$, the graph obtained from
the Petersen graph ${\cal P}_{10}$ by removing an edge $uv$ (Any
other graph $G$ with $\phi_C(G)=5$, that reduces to less than $5$
when an edge is removed, can be used instead of ${\cal P}_{10}$).
Since $\phi_c({\cal P}_{10})=5$, the capacity of ${\cal
P}^*_{10}(u,v)$ is disjoint from $(1,4)$ and therefore a subset of
$(4,1)$. On the other hand, $\phi_c({\cal P}^*_{10})<5$, which
implies that $0$ is included in the capacity. The only set in
$\mbox{\it SI}_5$ which meets these two conditions is $(4,1)$.
More sets of $\mbox{\it GI}_5$ can now be generated:
\begin{itemize}
\item $(4,1)$ of amplitude and measure $2$.
\item $\emptyset=(1,4)\cap(4,1)$ of amplitude and measure $0$
\item $(3,2)=(4,1)+(4,1)$ of amplitude and measure $4$
\item $(0,0)=(4,1)+(1,4)$ of amplitude and measure $5$
\item $(4,0)\cup(0,1)=(4,1)\cap(0,0)$ of amplitude and measure $2$
\item $(3,0)\cup(0,2)=(0,0)\cap(3,2)$ of amplitude and measure $4$
\item $(1,2)\cup(3,4)=(3,2)\cap(1,4)$ of amplitude $3$ and measure
$2$
\item
$\mathbb{R}/5\mathbb{Z}-\{1,4\}=((1,2)\cup(3,4))+((1,2)\cup(3,4))$
of amplitude and measure $5$
\item
$\mathbb{R}/5\mathbb{Z}-\{0,1,4\}=(\mathbb{R}/5\mathbb{Z}-\{1,4\})\cap
(0,0)$ of amplitude and measure $5$
\item $(3,2)-\{1,4\}=(3,2)\cap(\mathbb{R}/5\mathbb{Z}-\{1,4\})$ of amplitude and measure $4$
\item
$(3,2)-\{0,1,4\}=(3,2)\cap(\mathbb{R}/5\mathbb{Z}-\{0,1,4\})$ of
amplitude and measure $4$
\end{itemize}
 Once again, the serial-parallel routine ceases to produce new results and an additional
 tool is required.\\
 Consider a vertex $x$ of degree $3$ where one of
the three edges incident with $x$ is of capacity
$B\subseteq(1,4)$. Let  the orientation of the other two edges $a$
and $b$ be such that one of them is outgoing from $x$ and the
other one is ingoing. Let $f$ be a sub-$5$-{\sc mcnzf}. For the
difference between $f(a)$ and $f(b)$ to lie in $B\subseteq(1,4)$,
these two values cannot belong to the same unit interval. As a
consequence:
\begin{lemma} \label{measure2}
Let $P$ be a path with at least one internal vertex in a graph
$G$, such that: all edges of $P$ are of the same capacity $A$, of
measure $\mbox{Me}(A)=2$; every internal vertex $v$ has degree
three and the third edge incident to $v$ (the one not in $P$) is
of some capacity $B_v\subseteq(1,4)$ (in particular a simple
edge). Assume an orientation of $G$ where $P$ is a directed path.
If $f$ is a sub-$5$-{\sc mcnzf} in $G$, then the values of $f$
along $P$ are alternating between the two unit intervals contained
in $A$.
\end{lemma}
Alternating values along an odd cycle bear a contradiction, which
implies the following two conclusions:

\begin{corollary}\label{odd1}
 Let $C$ be an odd cycle in a graph $G$, along vertices of degree
 $3$. If all edges of $C$ are of the same capacity $A$, of measure
$\mbox{Me}(A)=2$, and the third edges incident with each vertex of
$C$ is of some (not necessarily the same) capacity
$B\subseteq(1,4)$ (in particular a simple edge), then
$\phi_c(G)\ge 5$.
\end{corollary}
and
\begin{corollary}\label{odd}
Let $C$ be an  odd cycle in a graph $G$, along vertices of
degree $3$, such that: all edges of $C$ are of the same capacity
$A$, of measure $\mbox{Me}(A)=2$, and the third edge incident with
each vertex of $C$ is of capacity $B\subseteq(1,4)$ (in particular
a simple edge). The deletion of an edge $uv$ of $C$ results in a
g-edge $q=G_{u,v}$ such that $\mbox{\it CP}_5(q)\subseteq
\overline{\sigma(A)}$.
\end{corollary}

Corollary \ref{odd} now allows us to construct a $3$-edge-connected
$(1,4)$-edge, to be later used as a replacement for single
edges, when higher connectivity is required:
\begin{definition}
A {\bf thick} (1,4)-edge with terminals $u$ and $v$ is obtained
from a copy of $K_4$ where two edges of a triangle are replaced by
$(4,1)$-edges, and the third edge, $uv$, of that triangle is
removed.
\end{definition}

By Corollary \ref{odd}, the capacity of the obtained g-edge with
terminals $u$ and $v$ (see Figure \ref{gedges}, top right) is a
subset of $\overline{\sigma(4,1)}=(1,4)$. It is easy to verify
that $2$ belongs to the obtained capacity, which is therefore,
indeed $(1,4)$ (no set in $\mbox{\it SI}_5$ which is a proper
subset of $(1,4)$ includes the point $2$).

\smallskip

Following the exact same lines with the capacity $(4,1)$ replaced
by $(1,2)\cup(3,4)$, Corollary \ref{odd} can be used to further
broaden the list of sets in $\mbox{\it GI}_5$. If two edges of a
triangle of $K_4$ are replaced by $(1,2)\cup(3,4)$-edges, and the
third edge $uv$ of that triangle is removed, then, by Corollary
\ref{odd}, the capacity $D$ of the obtained g-edge with terminals
$u$ and $v$ (see Figure \ref{gedges}) is a subset of
$\overline{\sigma((1,2)\cup(3,4))}=(4,1)\cup(2,3)$. It is easy to
verify $0\in D$ and $\frac52 \in D$, so $D=(4,1)\cup(2,3)$.

New members can now be added to $\mbox{\it GI}_5$
\begin{itemize}
\item$(4,1)\cup(2,3)$, of amplitude $4$ and measure $3$
\item $(4,0)\cup(0,1)\cup(2,3)=((4,1)\cup(2,3))\cap(0,0)$, of same
amplitude and measure
\item$(2,3)=((4,1)\cup(2,3))\cap(1,4)$, of amplitude and measure $1$

\end{itemize}
We have listed, so far, $16$ members of $GI_5$. The remaining $5$
sets in $SI_5$ are obtained by removing $\{2,3\}$ from the $5$
sets in our list that contain $\{2,3\}$ as a subset. A similar
phenomenon was observed in $GI_3$. We tend to believe these $5$
sets are not graphic, but so far, we have nothing to support that
claim.

\medskip

\begin{figure}[htb]
\centering \includegraphics[width=14cm]{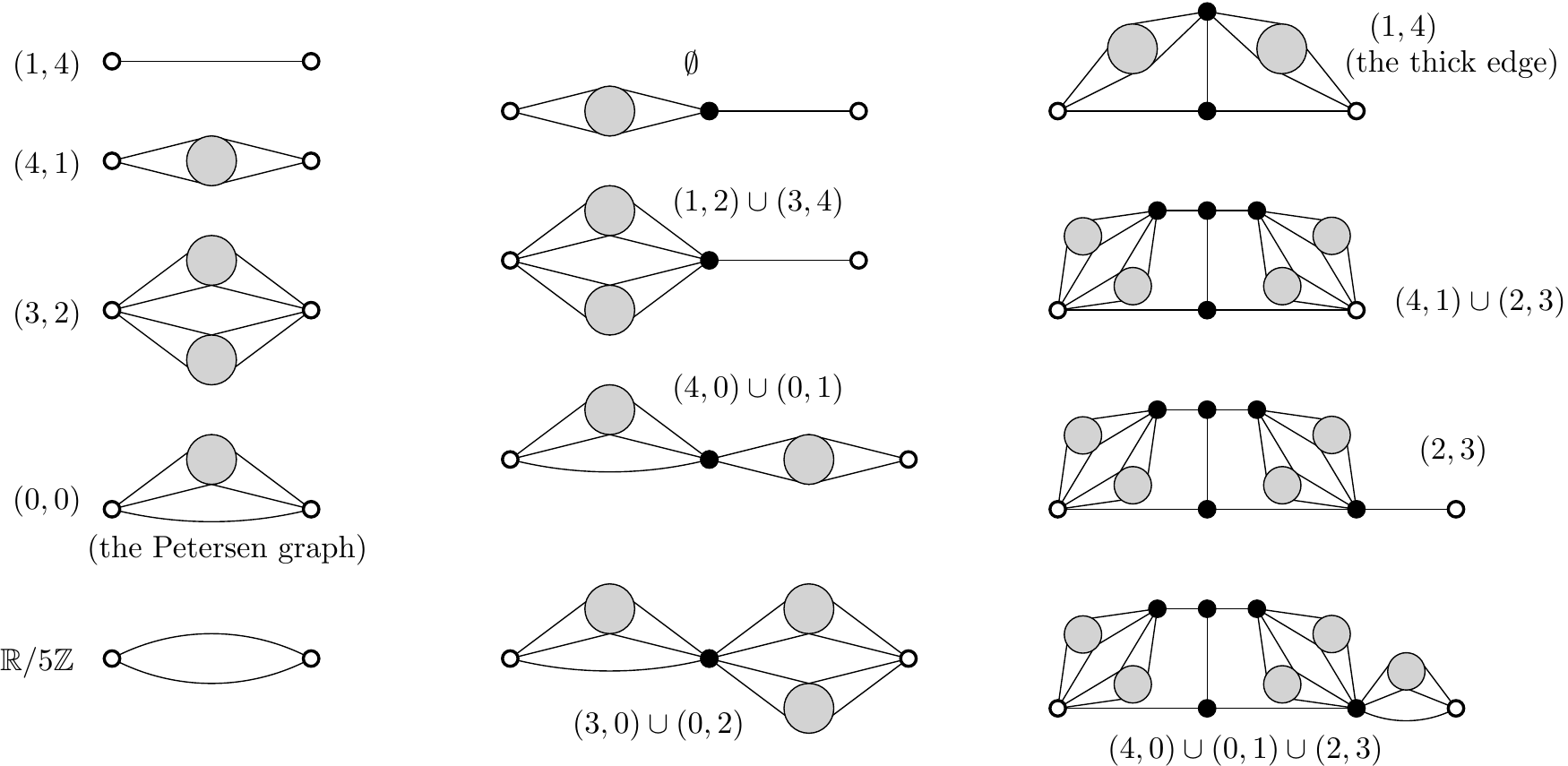} \caption{Basic
generalized edges}\label{gedges}
\end{figure}

Figure \ref{gedges} shows g-edges for twelve capacities (two for
(1,4)) out of the sixteen on our list. Next to the diagram of each
is its $5$-capacity. Terminal vertices are depicted by white dots.
The repeatedly used \includegraphics[width=1cm]{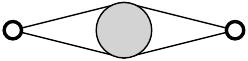} shaped
pattern stands for the $(4,1)$-edge ${\cal P}^*_{10}(u,v)$
obtained from ${\cal P}_{10}$ by the removal of an edge $uv$.

\medskip

We conclude our arsenal of $\mbox{\it GI}_5$ related observations
with the following schema:
\begin{lemma} \label{cycle} Let $C$ be a cycle, consisting of
simple edges, in a graph $G$. Let $T\subseteq \mbox{\it GI}_5$ be
a set of $5$-capacities, such that the amplitude of their union is
at most $3$. Let $G'$ be obtained from $G$ by replacing every edge
of $C$ by an $A$-edge, for some (not necessarily the same) $A\in
T$. If $\phi_c(G)\ge 5$, then also $\phi_c(G')\ge 5$. \end{lemma}
\begin{proof}
Assume, to the contrary that there exists a sub-$5$-{\sc mcnzf}
$f$ in $G'$. Since the union of all members of $T$ is of amplitude
at most $3$, there exists $t\in \mathbb{R}/5\mathbb{Z}$ such that
for every $y\in \bigcup_{A\in T}A$,  $~y+t\in (1,4)$. Let $f_1$ be
a flow in $G'$, defined by $f_1(e)=t$ if $e$ belongs to $C$, and
$f_1(e)=0$ for all other edges $e$. The flow value of $f+f_1$
belongs to $(1,4)$ for every edge of $C$ and $f+f_1$ is identical
to $f$ on the other edges. $f+f_1$ is therefore, a sub-$5$-{\sc
mcnzf} in the original graph $G$, which is a contradiction if
$\phi_c(G)\ge 5$.
\end{proof}

\section{Explicit construction of graphs $G$, in particular snarks,
with $\phi_c(G)\ge5$}\label{explicit}
 Let $F_{\ge 5}$ stand for the set of graphs
$G$ with $\phi_c(G) \ge 5$, and let $S_{\ge 5}$ be the set of all
snarks in $F_{\ge 5}$. Recall that a snark is  a 3-regular graph
$S$, cyclically $4$-edge-connected, of girth $5$ or more and
circular flow number $\phi_c(S)>4$.

For some time, $S_{\ge 5}$ was conjectured to consist solely of
the Petersen graph \cite{mohar}, until an infinite family of such
snarks was presented in \cite{mr}. Similar constructions, aimed
toward different goals, can also be found in other articles, e.g
\cite{Hag12}. Nonetheless, we now demonstrate that $S_{\ge 5}$ is
in fact much richer than that. The concepts and tools, developed
on the previous section, give rise to a large variety of snarks in
$S_{\ge 5}$.
 Let us first note that a graph $G \in F_{\ge 5}$, of the right
 girth and connectivity, can be transformed into a snark in $S_{\ge
 5}$, by means of:

 \begin{definition}
Given a graph $G$, an \textbf{expansion} of a vertex $x$ into a
graph $X$ is obtained by: Deleting the vertex $x$ from $G$ and
replacing it by the graph $X$. Each edge $yx$ of $G$ is replaced
by an edge between $y$ and an arbitrary vertex of $X$.
\end{definition}

\begin{proposition} \label{exp}
 Let $G'$ be obtained by a vertex expansion of a graph $G=(V,E)$
 and let $f$ be a flow in $G'$, then the restriction of $f$ to $E$
 is a flow in $G$. Consequently, $\phi_c(G')\ge \phi_c(G)$. In particular,
 if $G \in F_{\ge 5}$, so is $G'$.
 \end{proposition}

Expansion can be accompanied by:
\begin{definition}
\textbf{Smoothing} a vertex $x$ of degree $d(x)=2$ means the
removal of $x$, while merging the two edges $ux$ and $xv$ incident
with $x$, into a single new edge $uv$.
\end{definition}

Observe that smoothing a vertex does not affect flow values.

\smallskip

Considering the above, we focus on constructing cyclically
$4$-edge-connected graphs in $F_{\ge 5}$, of girth at least $5$.
Each such graph can then be transformed into infinitely many
snarks $S \in S_{\ge 5}$, by selecting proper expansion graphs
(almost) arbitrarily. An expansion graph $X$ should not
necessarily be highly connected. Edges by which $X$ is attached to
the rest of the graph can be tailored to repair small edge-cuts.
The graph $X$ is not even required to be connected (see Figure
\ref{small}). Let us now briefly describe some actual
constructions, based on these principles :
\subsection{Constructions based on Corollary \ref{odd1}}
Corollary \ref{odd1} lets us turn any arbitrarily selected
non-bipartite cubic graph $G$, which is ``almost" cyclically
$4$-edge-connected, of girth ``almost" $\ge 5$ into a snark, by
replacing the edges of any odd cycle $C$ by $A$-edges, of measure
Me$(A)=2$, and then properly expanding vertices of degree $>3$. We
used the term ``almost", because the length of  $C$ can be less
than $5$, yet larger than that, when simple edges are replaced by
g-edges. The same holds for small edge-cuts, which includes edges
of $C$. There are three different sets of measure $2$ in
$\mbox{\it GI}_5$. The initial graph $G$ and expansion graphs for
vertices of large degrees can be arbitrarily selected. The family
of obtained members of $S_{\ge 5}$, by means of this method only
is already rather diverse and rich.

\smallskip

Here are some of the smallest possible examples: Replace the three
edges of a triangle in $K_4$, by $(4,1)$-edges, ${\cal P}^*_{10}(u,v)$,
to obtain the graph drawn in Figure \ref{small}, left. By Corollary
\ref{odd1} the obtained graph belongs to $F_{\ge 5}$. It includes
$3$ vertices of degree $5$, which should be expanded in order to
obtain a snark.

\medskip

\begin{figure}[htb]
  \centering
  \includegraphics[width=4cm]{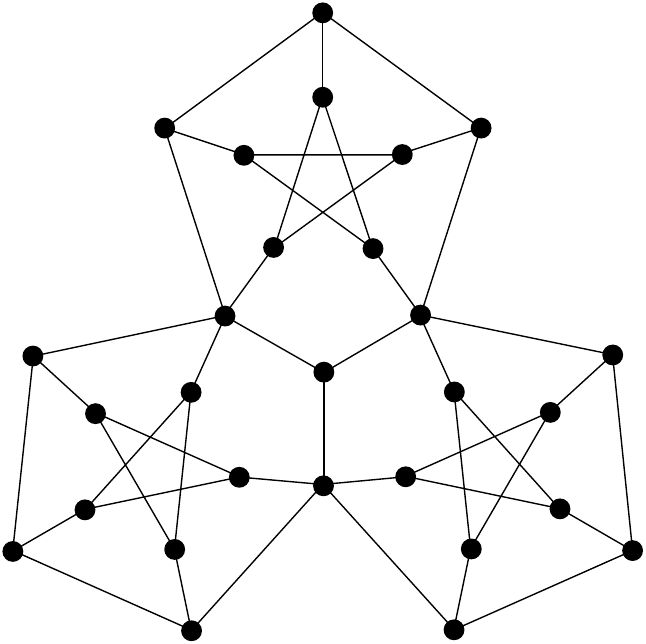}\hspace{0.5cm}
\includegraphics[width=4cm]{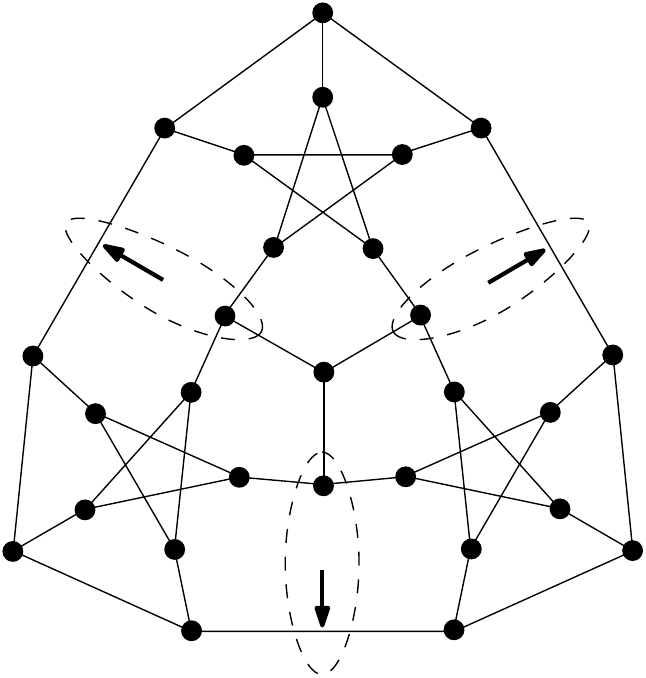}\hspace{0.5cm}
\includegraphics[width=4cm]{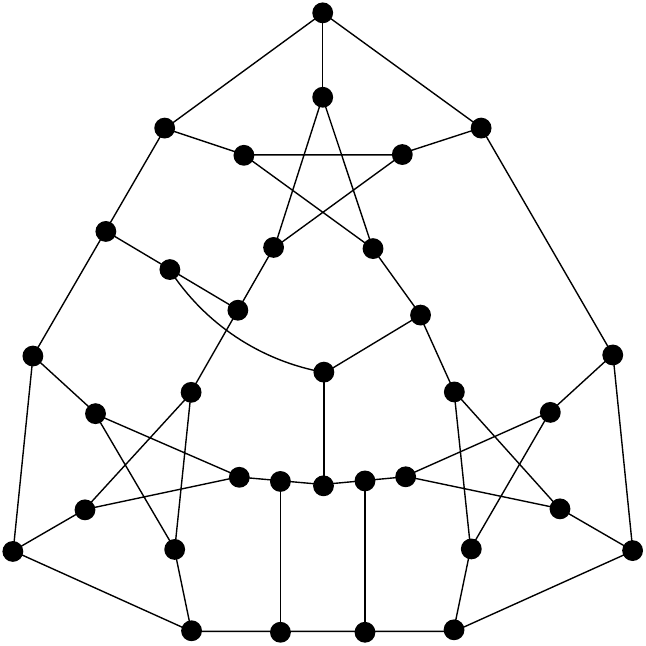}
\caption{$K_4$ with a triangle consisting of $(4,1)$-edges (left); $S_{28}$ (center); and three distinct expansion graphs
applied to the same source graph (right).}\label{small}
\end{figure}

The snark in Figure \ref{small} (center) (Let's call it $S_{28}$,
not an official name) is obtained by expanding each of the $3$
vertices of degree $5$ of the graph in Figure \ref{small} (left)
into a graph which consists of two isolated vertices. Each arrow
in the diagram points at a location of a missing second vertex,
which was removed by \textbf{smoothing}. The graph $S_{28}$ was
identified as a snark during a computerized survey \cite{snarks},
to discover all snarks of order $30$ or less. The circular flow
number of the graph was calculated by M\'a\v cajov\'a, and Raspaud
\cite{mr} -- also by means of a computer program. They identified
$S_{28}$ as the smallest snark in $S_{\ge5}$, other than ${\cal
P}_{10}$. The graph does not comply with the construction method
developed in \cite{mr}.

Two more snarks in $S_{\ge 5}$ where discovered by means of a
computer, both of order 30. These are obtained by replacing one of
the expansion graphs in $S_{28}$ by a path $P_3$. There are indeed
two non-isomorphic patterns to do so with cyclically
$4$-edge-connectivity retained. $P_3$ is one of the expansion
graphs of the snark on Figure \ref{small} (right).

\subsection{Constructions based on Lemma
\ref{cycle}}
\label{oncycle} This method seems similar to the one
obtained from Corollary \ref{odd1}. That similarity is somewhat
misleading. Here the initial graph $G$ is not arbitrary. The selection
of $G$ is restricted to previously constructed members of $F_{\ge
  5}$. On the other hand, this method is not limited to odd cycles and
the set $T$ of replacement capacities is richer. $T$ can either
contain a single capacity of amplitude $3$ or less, and there are six
distinct such capacities in our list of members of $\mbox{\it GI}_5$,
or it can consist of more than one set, such as
$T=\{(4,0)\cup(0,1),(4,1)\}$, $T=\{(1,2)\cup(3,4),(2,3),(1,4)\}$ and
many more. Needless to say, expansion graphs can be freely chosen
while applying this method, just as well. Lemma \ref{cycle}, because
it is weaker than Corollary \ref{odd1}, need to be recursively applied
to previously constructed graphs $G \in S_{\ge 5}$. Such a recursion
would be redundant for Corollary \ref{odd1}, where any initial graph
$G$ can be selected to start with. Yet, new members of $S_{\ge 5}$ can
replace the Petersen graph in producing distinct g-edges for the same
capacities, which opens many routes for multi-dimensional recursion
using both methods. Let us remark that the construction schema of
\cite{mr} is obtained from Lemma \ref{cycle}, starting initially with
the Petersen graph, using ${\cal P}^*_{10}(u,v)$ for edge replacement,
two isolated vertices (see Figure \ref{small}, center) as the only
expansion graph, and recursively applying the same technique to the
obtained graphs.

 \subsection{Various g-edges with the same $5$-capacity}
 \label{ongedges}
Quite obviously:
\begin{proposition} \label{phigk}
 Given an integer $k\ge 2$ and a graph $G$, $\phi_c(G)<k$ if and only if $0\in
\mbox{\it CP}_k(G_{u,v})$ for  a pair (equivalently
 all pairs) of vertices $u$ and $v$ of $G$.
\end{proposition}
Consequently, any $A$-edge where $0 \notin A$ is a graph in
$F_{\ge5}$, and, if it has the right girth and connectivity, it
can be turned into a snark in $S_{\ge5}$, via vertex expansion. In
our list of sixteen sets from $\mbox{\it GI}_5$, there are {\bf
ten} that do not include $0$. Of the smaller g-edges representing
each capacity ($13$ are depicted in Figure \ref{gedges}), the only
one that meets the required cyclically $4$-edge-connectivity is
the Petersen graph, listed as a $(0,0)$-edge. Indeed ${\cal
P}_{10}\in S_{\ge5}$. It is cubic and no expansion is required
here. However, for each $A\in \mbox{\it GI}_5$, there are many
distinct $A$-edges of higher connectivity. One way to construct
such g-edges is the replacement of one, or more, simple edges, by
the $3$-edge-connected thick $(1,4)$-edge. Let us show one
detailed example: The {\bf Butterfly graph} (see Figure
\ref{butter}) is a $(1,4)$-edge with terminals $u$ and $v$. It is,
therefore, in $F_{\ge 5}$.  The graph is cyclically
4-edge-connected and can be turned into a (actually many) snark $S
\in S_{\ge 5}$,  by expansion of vertices of high degrees. Here is
how it is built: Starting with a thick $(1,4)$-edge, $G$ with
terminals $u$ and $v$, the subgraph $Q$, circled on the right
``wing" is a thick $(1,4)$-edge, which replaces a simple edge $xy$
of $G$. Similarly replaced is a simple edge $yt$ on the left wing.
Replacements of simple edges by any $(1,4)$-edges do not affect
the existence of sub-$5$-{\sc mcnzf}'s in a graph.

\medskip

\begin{figure}[htb]
\centering \includegraphics[width=8cm]{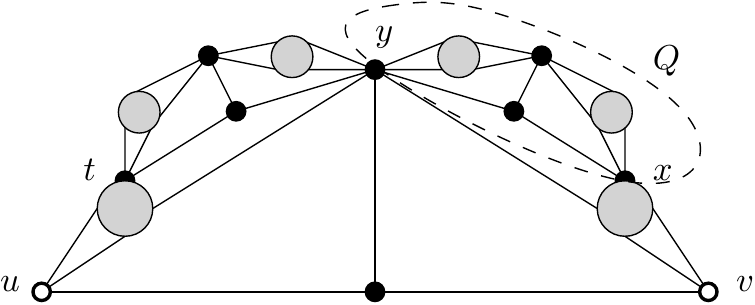} \caption{The
Butterfly $(1,4)$-edge.}\label{butter}
\end{figure}

Similarly, an unlimited set of alternative g-edges can be
constructed to represent each capacity $A\in \mbox{\it GI}_5$,
from which snarks $S \in S_{\ge 5}$ can be constructed, of all
shapes and sizes.

\medskip

Yet another simple observation: If an $A$-edge in a graph $G$ is
replaced by a $B$-edge with $B\subseteq A$, it does not give rise
to any sub-$5$-{\sc mcnzf} in the obtained graph $G'$, which does
not exist in $G$. In particular, if $G \in F_{\ge 5}$ so is $G'$.
It is easy to verify that, if an edge $uv$ is removed from a graph
$G \in F_{\ge
  5}$, then the capacity of the obtained g-edge $(G-uv)_{u,v}$ is either
$(4,1)$, or $\emptyset$. Considering the observation above, any such
g-edge can replace any $(4,1)$-edge to generate new members of $F_{\ge
  5}$, from previously generated ones.

\smallskip

One almost last peek into that seemingly vast bag of tricks: To
push any cubic graph (containing only simple edges) into $F_{\ge
5}$, in a single step, just replace two adjacent (simple) edges by
a pair of $(2,3)$ edges. The flow on the third edge incident with
their common endvertex is now restricted to $(2,3)+(2,3)=(4,1)$.
Yet, as that simple third edge is of capacity $(1,4)$, no
sub-$5$-{\sc cnzf} exists in the obtained graph, as promised. If
aiming toward snarks, the $(1,4)$ component of each of the
$(2,3)$-edges, should be thick, to provide the required
edge-connectivity.

\smallskip

We kept the simplest trick to the very end of the list: Insert a
single $\emptyset$-edge anywhere (make it thick to allow snarks).
No sub-$5$-flow is admitted anymore.

\subsubsection{Redundancy}\label{sec:red}
The construction methods described so far are by no means
independent. In fact, overlapping is rather wide and the same
graphs and snarks are generated in many ways. Here are a few of
many examples:

Let $(G-uv)_{u,v}$ be a g-edge, obtained by the removal of an edge
$uv$ from $G \in F_{\ge 5}$, as described in Section
\ref{ongedges}. If $uv$ belongs to an expansion subgraph $H$, used
as part of the construction of $G$, then any use of $(G-uv)_{u,v}$
as an edge, to construct a new graph is equivalent to replacing
$H$ by a larger expansion graph $H'$, when constructing $G$. Since
the selection of an expansion graph is arbitrary, $H'$ could have
been selected to start with. So using $(G-uv)_{u,v}$ as an edge
replacement is entirely redundant here.

Similarly, let $G'$ be obtained from a graph $G$ by means of
Corollary \ref{odd1}. When applying Lemma \ref{cycle} to $G'$, if
the cycle $C$ is contained in $G$, then the entire new
construction can be considered as part of a larger initial graph,
selected instead of $G$. Using Lemma \ref{cycle} in that case is,
therefore, redundant.

When elaborating on the usage of Lemma \ref{cycle} in Section
\ref{oncycle}, we did count $T=\{(2,3)\}$ among the relevant
subsets of $\mbox{\it GI}_5$. However, as specified on the
penultimate paragraph of Section \ref{ongedges}, two consecutive
$(2,3)$-edges turn any initial graph into $F_{\ge 5}$. Nothing is
gained by starting from a graph already in $F_{\ge 5}$ and
replacing an entire cycle.

When carefully checking the details, the last trick of Section
\ref{ongedges} (in its ``thick'' version) is equivalent to the selection of an arbitrary
graph, for expansion of a vertex of degree $5$, in the graph
depicted in Figure \ref{small}, left.

Despite these (and many other) causes of redundancy, the part of
 $S_{\ge 5}$ established by methods which were developed along this
article, appears to be pretty rich. Significant support to that
claim is provided in the next section.

\section{NP-Completeness}\label{sec:npc}
Less than a decade ago, the only known snark with circular flow
number $5$  was the Petersen graph. An infinite family of such
snarks was presented in \cite{mr}, yet, as we demonstrated in
previous sections, the entire collection $S_{\ge 5}$ of snarks
with flow number $5$ (or more?) is much richer than that. We now
utilize the power of our construction methods to show that $S_{\ge
5}$ is, in fact, rich and complex enough, to facilitate an
NP-completeness proof. In other words, that set of snarks is rich
and dense enough, so that every instance $I$ of any Co-NP problem,
can be emulated, by a snark $S \in S_{\ge 5}$, whose size and
complexity are of polynomial order, in comparison with the size of
$I$.
\begin{theorem} \label{snpc} Given an input snark $G$, deciding if
$\phi_c(G)<5$ is NP-Complete.
\end{theorem}
We first prove a somewhat weaker version, where no restrictions,
related to connectivity, girth, or vertex degrees are imposed on
the input graph:

\begin{lemma}\label{NPC}
 Given any input graph $G$, deciding if
$\phi_c(G)<5$ is NP-Complete.
\end{lemma}
\begin{proof}
A 3-hypergraph $H$ consists of a finite set $X$ of {\bf nodes} and
a collection of three element subsets of $X$, called {\bf
triplets}. The following problem is known to be NP-complete: Given
a 3-hypergraph $H$, can the set of nodes be partitioned into two
``color" sets, so that no triplet contains three nodes of the same
color. If such a partition exists, then $H$ is said to be {\bf
2-colorable}.

Given a 3-hypergraph $H$, we show that a graph $G(H)$ can be
constructed, in polynomial time, such that $\phi_c(G(H))<5$ if
and only if $H$ is $2$-colorable.

\bigskip

\noindent {\bf Constructing $G(H)$.} Each node $x$ is represented by a {\bf node-cycle}
$C(x)$. The length of $C(x)$
 is twice the number of triplets that contain $x$. The edges of $C(x)$ are all of capacity
 $(1,2)\cup(3,4)$. The vertices along $C(x)$ are alternately referred to as
 positive and negative terminals of $C(x)$.
 \noindent

 Each triplet $T$ is represented by a {\bf triplet-cycle} $C(T)$, which consists of
 six simple $(1,4)$-edges, and six vertices $(T^+_1,T^+_2,T^+_3,T^-_3,T^-_2,T^-_1)$ in that
 (cyclic) order.

 Every occurrence of a node $x$ in a triplet $T$ is represented by
 two {\bf connector edges} of capacity $(1,2)\cup(3,4)$, one
between a positive terminal of $C(x)$ and a vertex $T^+_i$ of
$C(T)$, and the other one between a negative terminal of $C(x)$
and
 the vertex $T^-_i$ (same index $i$) of $C(T)$. Each vertex of
every node-cycle and every triplet-cycle is the endvertex of
exactly one connector.

\smallskip

 We choose, for reference, an orientation of
$H$ where all node-cycles and triplet-cycles are directed
cyclically, and the connectors are directed from node-cycles
toward triplet-cycles.

We split the main statement into two separate propositions.
\begin{proposition}\label{c2f}
If $H$ is $2$-colorable then $\phi_c(G(H))<5$.
\end{proposition}
\begin{proof}
Let $(X_1,X_2)$ be a partition of $X$ which yields a $2$-coloring
of $H$, that is, every triplet $T$ includes one node from one set
of the partition and two nodes from the other set.

To define a sub-$5$-{\sc mcnzf} $f$ in $G(H)$: We select a positive
number $0<\epsilon<1/6$ and assign $f$ values to the edges along
every node-cycle, alternately $2-\epsilon$ and $-(2-\epsilon)$.
Consequently, the $f$ values of the connectors incident with the
terminals of each node-cycle are alternately $t$ and $-t$, where
$t=1+2\epsilon$.

If $x \in X_1$, then we set $f(p)=t$, for every connector edge $p$
incident with a positive terminal of $C(x)$, and $f(n)=-t$ for
every connector $n$ incident with a negative terminal.

Conversely, if $x \in X_2$, then $f(p)=-t$, for every connector
edge $p$ incident with a positive terminal of $C(x)$, and $f(n)=t$
for every connector $n$ incident with a negative terminal. Observe
that $\pm(2-\epsilon)$ and $\pm(1+2\epsilon)$ indeed belong to the
capacity $(1,2)\cup(3,4)$.

It remains to define $f$ on the edges of the triplet cycles.
Consider a triplet $T$. Since $(X_1,X_2)$ yields a 2-coloring, the
flow $f$ equals $t$ on one of the three connectors incident with
$T^+_1,T^+_2,T^+_3$, and $f$ equals $-t$ on the other two, or vice
versa. Either way, the (cyclic) sequence of flow values on the six
connectors incident with vertices of $C(T)$ is either
$(t,-t,t,-t,t,-t)$,  or $(t,t,-t,t,-t,-t)$, or obtained from the
second by reversing the order (Note that, as the sequence is
cyclic, it does not necessarily start at $T^+_1$, but at any
conveniently selected ``first" vertex).

We now assign $f$ value $t$ to the ``closing" edge of $C(T)$, that
is, the edge going from the ``sixth" vertex to the ``first" one on
the sequence, as listed above. The $f$ values of the edges along
$C(T)$ then become $(t,2t,t,2t,t,2t)$, or $(t,2t,3t,2t,3t,2t)$.
Recall that $t=1+2\epsilon$ with $0<\epsilon<1/6$. Consequently,
$t,2t$ and $3t$ are all valid flow values in $(1,4)$. That remains
true also if the order is reversed and the obtained values become
$-t,-2t$ and $-3t$. For every edge $e$ of $G(H)$, the flow $f$
satisfies $f(e)\in \mbox{\it CP}_5(e)$, so it is indeed a
sub-$5$-{\sc mcnzf} in $G(H)$.
\end{proof}

\begin{proposition}\label{f2c}
If $\phi_c(G(H))<5$ then $H$ is $2$-colorable.
\end{proposition}
\begin{proof}
Let $f$ be a sub-$5$-{\sc mcnzf} in $G(H)$. By Lemma
\ref{measure2}, the flow values on the edges along each node-cycle
$C(x)$ alternately belong to $(1,2)$ and to $(3,4)$. Consequently,
the values on the connectors incident with terminals of $C(x)$
alternately belong to $(1,3)$ and to $(2,4)$. As the capacity of a
connector is $(1,2)\cup(3,4)$, the actual values alternate between
$(1,2)$ and $(3,4)$. Let $X_1$ be the set of nodes $x$, for which
$f(p)\in (1,2)$ on the connectors $p$, incident with positive
terminals of $C(x)$, and let $X_2$ be the set of nodes $x$, for
which $f(p)\in (3,4)$ on the connectors $p$, incident with
positive terminals of $C(x)$. We claim that $(X_1,X_2)$ yields a
$2$-coloring of $H$. Assume, to the contrary, that this is not the
case, then there exists a triplet $T$ such that the three $f$
values $d_1,d_2,d_3$ on the connectors incident with
$T^+_1,T^+_2,T^+_3$, all belong to the same unit interval, say, to
$(1,2)$. In that case, at least one of the four $f$ values,
$a,a+d_1,a+d_1+d_2,a+d_1+d_2+d_3$ (on the four consecutive edges
of $C(T)$, starting with $f(T^-_3T^+_1)=a$), belongs to $(4,1)$,
in contradiction to $f$ being a sub-$5$-{\sc mcnzf}. The same
contradiction holds if $d_1,d_2$ and $d_3$ belong to $(3,4)$.
\end{proof}
Propositions \ref{c2f} and \ref{f2c} are combined to yield Lemma
\ref{NPC}.
\end{proof}

We are now set to prove the stronger Theorem \ref{snpc}.

\begin{proof}
Theorem \ref{snpc} is proved by converting the graph $G(H)$,
described in the proof of Lemma \ref{NPC}, in polynomial time,
into a snark $S(H)$, such that $\phi_c(S(H))<5$ if and only if
$\phi_c(G(H))<5$.

As a snark, $S(H)$ should be of the right girth and connectivity.
We should be, therefore, more specific about the structure of
$G(H)$. Recall that a $(1,2)\cup(3,4)$-edge is a serial join of a
$(3,2)$-edge and a $(1,4)$-edge. We use the minimal $(3,2)$-edge
and the thick $(1,4)$-edge (see Figure \ref{gedges}, top right)
for the edges along each node-cycle. For connectors, on the other
hand, we use a simple edge as the $(1,4)$ component. Also, of the
two terminals of a connector, the one which belongs to the simple
edge is selected to be attached to the triplet-cycle. That way,
every vertex of every triplet-cycle is incident with three simple
edges and it is of genuine (not ``generalized") degree $3$. A part
of a node cycle and its incident connectors is depicted in
Figure~\ref{link}. It is easily verified that $G(H)$ is now
cyclically $4$-edge-connected and of girth at least $5$. As a
snark, the circular flow number of $S(H)$ should be larger than
$4$. To guarantee that, we first prove:
\begin{proposition}
Regardless of the 3-hypergraph $H$, being $2$-colorable or not,
$\phi_c(G(H))>4$.
 \end{proposition}
 \begin{proof}
 Recall the definition of  ${\cal P}^*_{10}(u,v)$, as the g-edge obtained from the Petersen graph ${\cal
 P}_{10}$, by the removal of an edge $uv$. The range of a $4$-{\sc mcnzf} is the closed interval $[1,3]$ modulo
 $4$. The Petersen graph ${\cal P}_{10}$ does not
 admit a $4$-{\sc mcnzf}. As a result, the set of modulo $4$ flow values that can be
 ``pushed" through ${\cal P}^*_{10}(u,v)$ with the flow on each edge restricted to
 $[1,3]$ is disjoint from $[1,3]$. Since that set is a union of {\bf closed} unit
 intervals of $\mathbb{R}/4\mathbb{Z}$ (following the same argument that leads to
 Lemma \ref{cpk}), it consists solely of the point $0$. The same holds
 for the parallel join of two such g-edges, which forms the
 $(3,2)$ component of each connector of $G(H)$. As a result, a
 modular-$4$-flow on the simple edge component of a connector is also
 restricted to $0$, so no $4$-{\sc mcnzf} is possible.
 \end{proof}

It remains to take care of vertices $v$ with degree $d(v)>3$. That is
achieved, by an expansion of each such vertex $v$, into a cubic
subgraph. By Proposition \ref{exp}, an expansion of a vertex never
decreases the circular flow number. The tricky part is to guarantee
$\phi_c(S(H))<5$ whenever $\phi_c(G(H))<5$. Luckily, the structure of
each node-cycle is highly symmetric and it is basically the same for
every node-cycle and every graph $G(H)$.  Figure \ref{link}
illustrates the basic component of which $G(H)$ is made -- A
$(1,2)\cup(3,4)$-edge of a node-cycle and the incident connector. We
refer to that subgraph as a {\bf link}. All links in all graphs $G(H)$
are isomorphic. A link includes four vertices of degree larger than 3,
labeled on Figure \ref{link} as $x$ of degree $11$, $y$ of degree $7$,
$z$ of degree $5$ and $w$ of degree $5$. Vertices $x$, $y$, $z$, and
$w$ will be respectively replaced by four expansion graphs $X$, $Y$,
$Z$ and $W$.

\begin{figure}[htb]
\centering \includegraphics[width=12cm]{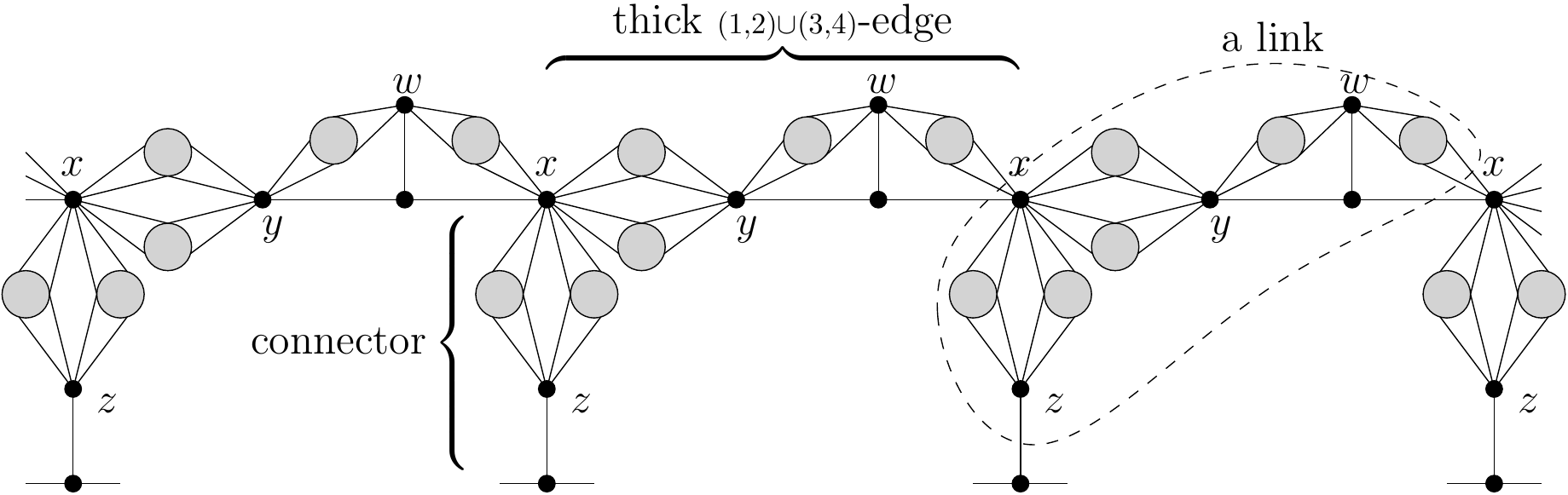} \caption{A segment
of a node-cycle with incident connectors}\label{link}
\end{figure}

 Consider the sub-$5$-{\sc mcnzf} $f$ defined in the proof of Proposition
 \ref{c2f}. After selecting $0<\epsilon<1/6$, say,
 $\epsilon=1/12$, $f$ is fully defined, up to switching the sign
 of all $f$-values on a node cycle and the incident connectors (by
 switching the roles of its positive and negative terminals). The flow $f$ is then defined for every edge
 of a link, and it is the same (up to reversing all signs) for
 all links. The expansion graphs $X$, $Y$, $Z$ and $W$ should be
 designed to allow a sub-$5$-{\sc mcnzf}, which maintains the $f$ values on
 all original edges of the link. An expansion graph suitable for
 that task, also clearly fits if all flow values on the incident
 edges switch their signs. Consequently, all we need are {\bf four
 fixed} expansion graphs $X$, $Y$, $Z$ and $W$, to be used for all links
 and convert every graph $G(H)$ into a snark $S(H)$, as required,
 in linear time.

 Explicit construction and presentation of these four graphs are certainly doable.
 However, that somewhat tedious mission, can be spared here. It suffices to show that such expansion
 graphs do exist, and that claim was already
 proved, in a more general setting.

In \cite{lokut}, Sections $3$ and $4$, the authors present a
procedure to convert a cyclically $4$-edge-connected graph $G$, of
girth at least $5$ and such that $4<\phi_c(G)<5$, into a snark $S$
with the same circular flow number. Starting with a given $r$-{\sc
cnzf} $f$ in $G$, they remove each vertex $v$ of degree $d(v)>3$
and reroute the edges incident with that vertex, through a certain
``network", which makes the obtained graph $S$ a snark. The
structure of each such network is derived from the values of $f$
on the edges incident with $v$.
 A new $r$-{\sc cnzf} is defined in $S$, whose restriction to the original edges
 of $G$ is the given flow $f$.
 The following lemma summarizes that part of \cite{lokut}, although it is
not explicit in that article:
\begin{lemma}\label{lokutllema}
 Let $g$ be an $r$-{\sc cnzf} in a cyclically $4$-edge-connected graph $G$,
 of girth at least $5$ and such that $4<\phi_c(G)<5$. By means of
 expansion of each vertex of degree larger than $3$ (possibly
 accompanied by smoothing of vertices of degree $2$), a snark $S$ is
 obtained from $G$. There exists an $r$-{\sc cnzf} $g_1$ in $S$, such that
 $g_1(e)=g(e)$ for every edge $e \in E(G)$.
\end{lemma}

Lemma \ref{lokutllema} relates to a (real valued) $r$-{\sc cnzf}. For
our purpose, it should be reformulated in terms of modular flows.

Let $f$ be an $r$-{\sc mcnzf} in (a reference orientation $D_0$ of) $G$.
By Proposition \ref{mod}, there exists an $r$-{\sc cnzf} $g$, on an
orientation $D$ of $G$, such that, for every edge $e$ of $G$,
$g(e)\equiv f_D(e) \mbox{ modulo } r$.

 If $G$ is cyclically $4$-edge-connected, of girth at least $5$ and
such that $4<\phi_c(G)<5$, then Lemma \ref{lokutllema} asserts the
existence of expansion graphs, one for every vertex $v$ of degree
$d(v)>3$, which turn $G$ into a snark $S$ with an $r$-{\sc cnzf}
$g_1$, whose restriction to $E(G)$ equals $g$. Let $\tilde{f}=(g_1
\mbox{ modulo } r)$. That is, $\tilde{f}$ is an $r$-{\sc mcnzf} in
$S$, obtained from $g_1$ by interpreting each $g_1(e)$ as a
residue in $\mathbb{R}/r\mathbb{Z}$. The restriction of
$\tilde{f}$ to $G$ is clearly $f_D$. The orientation does not
really matter. For simplicity's sake we can switch back to the
reference orientation $D_0$ and the flow $f$.

In summary:
\begin {lemma}
\label{lokutmod} Let $f$ be an $r$-{\sc mcnzf} in a cyclically
$4$-edge-connected graph $G$, of girth at least $5$ and such that
$4<\phi_c(G)<5$.
By means of expansion of each vertex of degree larger than $3$ (possibly accompanied by smoothing of vertices
of degree $2$), a snark $S$
is obtained from $G$. There exists an $r$-{\sc mcnzf} $\tilde{f}$ in $S$,
such that $\tilde{f}(e)=f(e)$ for every edge $e \in E(G)$.
\end{lemma}

Clearly, the structure of each expansion graph for a vertex $v$
depends solely on the values of $f$ on the edges incident with
$v$. The existence of the expansion graphs $X$, $Y$, $Z$ and $W$
required for our construction, immediately follows.

Let us reemphasize that the converting procedure of \cite{lokut}
is not a part of our reduction algorithm and its details and
complexity are irrelevant. It is used here only as an existence
assertion, to save an explicit presentation of the four expansion
graphs.

$S(H)$ is constructed by replacing each vertex $x$, $y$, $z$ and
$w$, in every link of $G(H)$, by expansion graphs, $X$, $Y$, $Z$
and $W$, all in linear time.
\end{proof}

\subsection{Analogous results for $r \in (4,5)$}
As previously stated, using open capacities seems to be a must
when dealing with sub-$5$-flows. However, it may also come in
handy when smaller values of $r$ are considered.

The Petersen graph ${\cal P}_{10}$ is known to  admit an (integer)
$5$-NZF $f$, such that $f(uv)=4$ for exactly one edge $uv$, and
$f(e) \in \{1,2,3\}$ for all other edges $e$. That, combined with
$\phi_c({\cal P}_{10})=5$, $\phi_c({\cal P}^*_{10})=4$ and
Proposition \ref{mod}, lead to the following (which is also stated
in \cite{pan} for the closed capacity):

For every rational $r \in (4,5)$ the closed $r$-capacity of ${\cal
P}^*_{10}(u,v)$ is $[4,r-4]$ (modulo $r$) and the open
$r$-capacity of the same g-edge is the open interval $(4,r-4)$
with the same boundaries. Lemma 3.5 of \cite{pan} states the
existence of a collection of g-edges, among them one with  closed
$r$-capacity $[r-1,1]$. Let this g-edge be denoted here by $Q_r$.
Following the proof of Lemma 3.5 in \cite{pan}, $Q_r$ is
constructed by repeatedly applying serial and parallel joins,
starting with copies of ${\cal P}^*_{10}(u,v)$ and simple edges
(of closed $r$-capacity $[1,r-1]$ and open $r$-capacity
$(1,r-1)$).

In general, open and closed capacities cannot be switched
carelessly. See the discussion following
Definition~\ref{capacity}. However, the serial-parallel technique, as well as the set sum and
intersection operations, result in analogous outcomes when applied
to open and closed intervals with the same boundaries.
Accordingly, the open $r$-capacity of $Q_r$ is $(r-1,1)$.

Replacing the g-edge ${\cal P}^*_{10}$ in our constructions, by
$Q_r$, allows the generalization of many of our results, from
$r=5$, to $r \in (4,5]$. In particular, meaningful results can be
obtained when using (even if implicitly) the analogues of Lemma
\ref{measure2} and Corollaries \ref{odd} and \ref{odd1}, which
strongly rely on the involved capacities being open capacities.

For example
\begin{theorem}
For every rational $r \in (4,5)$,  deciding whether
$\phi_c(S)<r$, for an input snark $S$ is NP-Complete.
 \end{theorem}

\begin{proof} {\bf (outlines)}
The proof accurately follows those of Lemma \ref{NPC} and Theorem
\ref{snpc}, subject to the following adaptations:
\begin{itemize}
\item $Q_r$ replaces every ${\cal P}^*_{10}(u,v)$ (the \includegraphics[width=1cm]{pattern} shape in
Figures \ref{gedges} and \ref{link}) of $G(H)$. Accordingly:
\item Edges of the node-cycles and the connectors are
of $r$-capacity $(1,2)\cup(r-2,r-1)$
\item The $(1,r-1)$ component of a connector is a thick
$(1,r-1)$-edge (see Figure \ref{gedges}, top right).
\item The parameter $\epsilon$ selected for the definition of $f$
in the proof of Proposition \ref{c2f}, should satisfy $0 <
\epsilon <(r-4)/6$.
\item In addition to the four vertices $x,y,z$ and $w$, there are
other vertices $v$ of degree $d(v)>3$ within every $Q_r$ subgraph.
However, the set of such vertices in a link is still finite and so
is the set of required expansion graphs. With a finite fixed set
of expansion graphs, the construction of $S(H)$ from $G(H)$ is
still performed in linear time.

\item Every detail of the proofs of Lemma \ref{NPC} and Theorem
\ref{snpc}, straightforwardly translates into this modified
setting.
\end{itemize}
\end{proof}

\section{Concluding remarks, open problems and directions for
further research}\label{sec:ccl}

\subsection{How actually rich is $S_{\ge 5}$?}

As diverse as they may seem, snarks constructed in this article
still share certain restricting characteristics. They are all
based on the $(4,1)$-edge ${\cal P}^*_{10}(u,v)$, obtained from
the Petersen graph, and, as such, they all contain ${\cal
P}^*_{10}(u,v)$ subgraphs, and they all (other than ${\cal
P}_{10}$) are no more than cyclically $4$-edge-connected.

High edge-connectivity is known to be correlated with smaller flow
number, so the above may hold for every $S \in S_{\ge 5}$, or even
every graph $G \in F_{\ge 5}$. We hesitate to call it a
``conjecture" (to replace Mohar's Strong $5$-flow Conjecture
\cite{mohar}). It is, however,  definitely an intriguing question,
which we are not the first to post (e.g. \cite{lokut}):
\begin{problem}
Is there a cyclically $5$-edge-connected graph $G \in F_{\ge 5}$,
other than ${\cal P}_{10}$?
\end{problem}

Or at least: \begin{problem}
 Is there a graph $G \in F_{\ge 5}$ with no ${\cal P}^*_{10}(u,v)$
 induced subgraph, possibly with each of the two terminals $u$ and
 $v$ split into two pairs $(u_1,u_2)$ and $(v_1,v_2)$, due to
 vertex expansion?
 Here ${\cal P}_{10}$, the parallel join of ${\cal P}^*_{10}(u,v)$ and a simple
edge is no exception.
\end{problem}

Positive answers to either one of the above, or to both, would
mean that the borders of $F_{\ge 5}$ and $S_{\ge 5}$ are far
beyond those drawn by our methods.

Assuming that this is not the case, then every $G \in F_{\ge 5}$
 is only cyclically $4$-edge-connected and contains a
${\cal P}^*_{10}(u,v)$ subgraph. Does this mean that the
construction tools developed in Sections \ref{>=5} and
\ref{explicit} can generate every member of $F_{\ge 5}$ and
$S_{\ge 5}$?

This last question is not entirely well defined. We did not
present a systematic list of construction methods, but only
sporadically demonstrated some. It might be beneficial to try and
make such a list, and then to study how extensively the various
methods overlap. Is it possible to define, or to get close to  a
``basis" of independent (or almost independent) constructing
operations, which can be combined to produce all other
constructing methods?

 We have, however, systematically
analyzed the set of graphic $5$-capacities $\mbox{\it GI}_5$. Here a well
defined question is in place:
\begin{problem}
Does the list of $5$-capacities, presented in Section \ref{>=5}, include all members of
$\mbox{\it GI}_5$?
\end{problem}

A similar question can be asked with regards to $\mbox{\it GI}_3$
(and $\mbox{\it GI}_4$, though we suggested no list for that one).

\subsection{3-poles - The two dimensional case}
Similar to a $g$-edge, a {\bf 3-pole} is a ``network" $H$ with
three terminals - a source $u$ and two sinks $v$ and $w$. As part
of a graph $G$, a $3$-pole $H$ is a subgraph, which shares its
three terminals, and is otherwise vertex-disjoint, from the rest
of the graph. The (open or closed) capacity of a $3$-pole is the
set of pairs $(x,y)$ of flow values, which can be simultaneously
pushed from $u$ to $(v,w)$: That is, $x$ from $u$ to $v$ and $y$
from $u$ to $w$, subject to given restrictions on the flow values
on edges of $H$. The open $k$-capacity of a $3$-pole, would then
be, as a generalization of Definition \ref{capacity}, a set of
points $(x,y)$ on the torus $(\mathbb{R}/k\mathbb{Z})^2$, obtained
that way, where the flow values on the edges of $H$ are restricted
to $(1,k-1)$ modulo $k$. Similarly to Lemma \ref{cpk}, the open
$k$-capacity of a $3$-pole is a symmetric (with respect to the
origin point $(0,0)$) union of open convex integer polygons on
$(\mathbb{R}/k\mathbb{Z})^2$. An integer polygon is a polygon
whose vertices have two integer coordinates. Two $3$-poles which
share a common source $u$ and two distinct pairs of sinks,
$(v_1,w_1)$ and $(v_2,w_2)$ with $k$ capacities $A$ and $B$, can
be merged to form parallel and serial joins. The capacity of a
parallel join is $A+B$. The capacity of a serial join, however, is
not the intersection but the {\bf composition} $A\circ \{
(-y,z)\,|\, (y,z) \in B\}$ of $A$ and $\{ (-y,z)\,|\, (y,z) \in
B\}$ as binary relations in $\mathbb{R}/k\mathbb{Z}$. Neither the
parallel join, nor the serial is uniquely defined. We leave the
details for further research. Anyway, there are finitely many
potential $k$-capacites to generate from $k^2$ integer points in
$(\mathbb{R}/k\mathbb{Z})^2$. Definitely, if $k$ is small, the
mission is within the reach of a computer assisted comprehensive
study.

\subsection{Other applications to Nowhere-Zero flow problems} Closed capacities of
 g-edges  were very successfully applied in
\cite{pan} to the study of $r$-{\sc cnzf} where $r<5$. However,
the authors only considered the serial-parallel mechanism and
there was no attempt to characterize the set of all these
capacities, similar to what we did for the set $\mbox{\it GI}_5$
of all open $5$-capacities.  Systematic development of the subject
may lead to further applications in the study of Nowhere-Zero
Flows. The same holds for $3$-poles and multi-poles of higher
dimension with respect to any $r \in (4,5]$ with closed, as well
as open capacities.

\subsection{Regular matroids and graph coloring}
 Lemma \ref{cpk} can be stated and proved in the wider setting of
 totally unimodular integer programming. As such, it can be
 applied to flows in general regular matroids, rather than just graphs.
 Particularly interesting may be the co-graphic case, where the
 analogue of an $r$-{\sc cnzf} is a tension function, induced by a proper $r$-circular coloring of a
 graph. $k$-Co-capacities, may appear to be useful for the study
 of circular graph coloring (where $r$ is not bounded by $5$ or
 $6$ or any other upper bound).

\end{document}